\documentclass[11pt, a4paper]{article}
\usepackage{amsmath}
\usepackage{amsthm}
\usepackage{amssymb}
\usepackage{graphicx}

\theoremstyle{plain}
\newtheorem{theorem}{Theorem}[section]

\newtheorem{lemma}[theorem]{Lemma}

\newtheorem{definition}[theorem]{definition}

\newcommand{\C}{\mathbb{C}}

\newcommand{\fig}[5]{
\begin{figure}[ht]
\begin{center}
\resizebox{#1}{#2}{\includegraphics{#3}}
\end{center}
\caption{#4}
\label{#5}
\end{figure}
}

\title{%
An $f$-chromatic spanning forest of edge-colored complete bipartite graphs%
}

\author{%
Kazuhiro Suzuki%
\footnote{This work was supported by MEXT. KAKENHI 21740085.}
\footnote{Department of Electronics and Informatics Frontier,
Kanagawa University, Yokohama, Kanagawa, 221-8686 Japan.
kazuhiro@tutetuti.jp.}
}

\date{\empty}

\begin{document}
\maketitle

\begin{abstract}
In 2001, Brualdi and Hollingsworth proved that
an edge-colored balanced complete bipartite graph $K_{n,n}$
with a color set $\C$ $=$ $\{1,2,3,$ $\ldots , 2n-1 \}$
has a heterochromatic spanning tree
if the number of edges colored with colors in $R$
is more than $|R|^2 /4$
for any non-empty subset $R \subseteq \C$,
where a heterochromatic spanning tree
is a spanning tree whose edges have distinct colors,
namely, any color appears at most once.
In 2010, Suzuki generalized heterochromatic graphs
to $f$-chromatic graphs,
where any color $c$ appears at most $f(c)$.
Moreover, he presented a necessary and sufficient condition
for graphs to have an $f$-chromatic spanning forest
with exactly $w$ components.
In this paper,
using this necessary and sufficient condition,
we generalize the Brualdi-Hollingsworth theorem above.
\\[6pt]
{\bf Keyword(s):}
$f$-chromatic,
heterochromatic,
rainbow,
multicolored,
totally multicolored,
polychromatic,
colorful,
edge-coloring,
spanning tree,
spanning forest.
\\[6pt]
\footnote{05C05 Trees.},
05C15\footnote{05C15 Coloring of graphs and hypergraphs.}.
\end{abstract}

\section{Introduction}

We consider finite undirected graphs without loops or multiple edges.
For a graph $G$, we denote by $V(G)$ and $E(G)$
its vertex and edge sets, respectively.
An \textit{edge-coloring} of a graph $G$
is a mapping $color:E(G) \rightarrow \C$,
where $\C$ is a set of colors.
An \textit{edge-colored graph} $(G, \C, color)$
is a graph $G$ with an edge-coloring $color$ on a color set $\C$.
We often abbreviate an edge-colored graph $(G, \C, color)$ as $G$.

An edge-colored graph $G$ is said to be \textit{heterochromatic}
if no two edges of $G$ have the same color,
that is,
$color(e_i) \ne color(e_j)$ for any two distinct edges $e_i$ and $e_j$ of $G$.
A heterochromatic graph is also said to be 
\textit{rainbow}, \textit{multicolored}, \textit{totally multicolored},
\textit{polychromatic}, or \textit{colorful}.
Heterochromatic subgraphs of edge-colored graphs
have been studied in many papers.
(See the survey by Kano and Li \cite{KaLi08}.)

Akbari \& Alipour \cite{AkAl06}, and Suzuki \cite{Su06}
independently presented a necessary and sufficient condition
for edge-colored graphs to have a heterochromatic spanning tree,
and they proved some results by applying the condition.
Here, we denote by $\omega(G)$ the number of components of a graph $G$.
Given an edge-colored graph $G$ and a color set $R$,
we define $E_{R}(G) = \{ e \in E(G) ~|~ color(e) \in R \}$.
Similarly, for a color $c$,
we define $E_c(G) = E_{\{c\}}(G)$.
We denote the graph $(V(G), E(G) \setminus E_R(G))$ by $G-E_R(G)$.

\begin{theorem}[Akbari and Alipour, (2006) \cite{AkAl06}, Suzuki, (2006) \cite{Su06}]
An edge-colored graph $G$
has a heterochromatic spanning tree
if and only if
\begin{equation*}
\omega(G-E_R(G)) \le |R|+1 \text{~~~~~ for any } R \subseteq \C.
\end{equation*}
\label{thm-Su06-1}
\end{theorem}

Note that if $R=\emptyset$ then the condition is $\omega(G) \le 1$.
Thus, this condition includes a necessary and sufficient condition
for graphs to have a spanning tree,
namely, to be connected.
Suzuki \cite{Su06} proved the following theorem
by applying Theorem \ref{thm-Su06-1}.

\begin{theorem}[Suzuki, (2006) \cite{Su06}]
An edge-colored  complete graph $K_n$ has a heterochromatic spanning tree
if $|E_c(G)| \le n/2$ for any color $c \in \C$.
\label{thm-Su06-2}
\end{theorem}

Jin and Li \cite{JiLi06} generalized Theorem \ref{thm-Su06-1}
to the following theorem,
from which we can obtain Theorem \ref{thm-Su06-1} by taking $k=n-1$.

\begin{theorem}[Jin and Li, (2006) \cite{JiLi06}]
An edge-colored  connected graph $G$ of order $n$
has a spanning tree with at least $k$ $(1 \le k \le n-1)$ colors
if and only if
\begin{equation*}
\omega(G-E_R(G)) \le n-k+|R| \text{~~~~~ for any } R \subseteq \C.
\end{equation*}
\label{thm-JiLi06-1}
\end{theorem}

If an edge-colored connected graph $G$ of order $n$
has a spanning tree with at least $k$ colors,
then
$G$ has a heterochromatic spanning forest with $k$ edges,
that is,
$G$ has a heterochromatic spanning forest with exactly $n-k$ components.
On the other hand,
If an edge-colored connected graph $G$ of order $n$
has a heterochromatic spanning forest with exactly $n-k$ components,
then we can construct a spanning tree with at least $k$ colors
by adding some $n-k-1$ edges to the forest.
Hence, we can rephrase Theorem \ref{thm-JiLi06-1} as the following.

\begin{theorem}[\cite{JiLi06}]
An edge-colored  connected graph $G$ of order $n$
has a heterochromatic spanning forest
with exactly $n-k$ components $(1 \le k \le n-1)$
if and only if
\begin{equation*}
\omega(G-E_R(G)) \le n-k+|R| \text{~~~~~ for any } R \subseteq \C.
\end{equation*}
\label{thm-JiLi06-2}
\end{theorem}

{\it Heterochromatic}
means that
any color appears at most {\it once}.
Suzuki \cite{Su10} generalized {\it once} to a mapping $f$
from a given color set $\C$ to the set of non-negative integers,
and introduced the following definition as a generalization
of heterochromatic graphs.

\begin{definition}
Let $f$ be a mapping from a given color set $\C$
to the set of non-negative integers.
An edge-colored graph $(G, \C, color)$ is said to be
\textit{$f$-chromatic}
if $|E_c(G)| \le f(c)$ for any color $c \in \C$. 
\label{def-Su10-2.1}
\end{definition}

Fig.\ref{fig-110405-1} shows an example
of an $f$-chromatic spanning tree of an edge-colored graph.
Let $\C = \{1,2,3,4,5,6,7 \}$ be a given color set of $7$ colors,
and a mapping $f$ is given as follows:
$f(1)=3$,
$f(2)=1$,
$f(3)=3$,
$f(4)=0$,
$f(5)=0$,
$f(6)=1$,
$f(7)=2$.
Then, the left edge-colored graph in Fig.\ref{fig-110405-1}
has the right graph as a subgraph.
It is a spanning tree
where each color $c$ appears at most $f(c)$ times.
Thus, it is an $f$-chromatic spanning tree.

\fig{0.85\textwidth}{!}{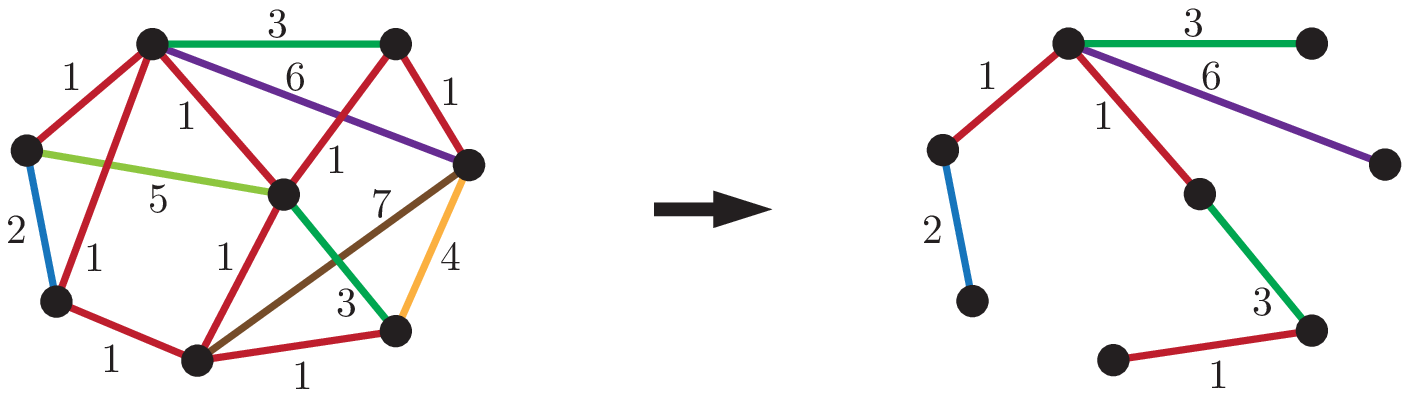}
{An $f$-chromatic spanning tree of an edge-colored graph.}
{fig-110405-1}

If $f(c) = 1$ for any color $c$,
then all $f$-chromatic graphs are heterochromatic
and also all heterochromatic graphs are $f$-chromatic.
It is expected many previous studies and results
for heterochromatic subgraphs
will be generalized.

Let $\C$ be a color set,
and $f$ be a mapping from $\C$ to the set of non-negative integers.
Suzuki \cite{Su10} presented the following necessary and sufficient condition
for graphs to have an $f$-chromatic spanning forest
with exactly $w$ components.
This is a generalization of Theorem \ref{thm-Su06-1} and Theorem \ref{thm-JiLi06-2}.

\begin{theorem}[Suzuki, (2010) \cite{Su10}]
An edge-colored graph $(G, \C, color)$ of order at least $w$
has an $f$-chromatic spanning forest with exactly $w$ components
if and only if
\begin{equation*}
\omega(G-E_R(G)) \le w+\sum_{c \in R}f(c) \text{~~~~~ for any } R \subseteq \C.
\end{equation*}
\label{thm-Su10-2.3}
\end{theorem}

By applying Theorem \ref{thm-Su10-2.3},
He generalized Theorem \ref{thm-Su06-2},
as follows.

\begin{theorem}[Suzuki, (2010) \cite{Su10}]
A $g$-chromatic graph $G$ of order $n$ with $|E(G)|>\binom{n-w}{2}$
has an $f$-chromatic spanning forest with exactly $w$ $(1 \le w \le n-1)$ components
if $g(c) \le \frac{|E(G)|}{n-w}f(c)$ for any color $c$.
\label{thm-Su10-2.5}
\end{theorem}

In this paper,
we will generalize the following theorem
for edge-colored complete bipartite graphs.

\begin{theorem}[Brualdi and Hollingsworth, 2001 \cite{BrHo01}]
Let $G$ be an edge-colored balanced complete bipartite graph $K_{n,n}$
with a color set $\C = \{1,2,3,$ $\ldots,$ $2n-1\}$.
Let $e_c$ be the number of edges with a color $c$,
namely, $e_c = |E_c(G)|$,
and assume that $1 \le e_1 \le e_2 \le \cdots \le e_{2n-1}$.
If $\sum_{i=1}^{r} e_i > r^2/4$ for any color $r \in \C$,
then $G$ has a heterochromatic spanning tree.
\label{thm-BrHo01}
\end{theorem}

\fig{\textwidth}{!}{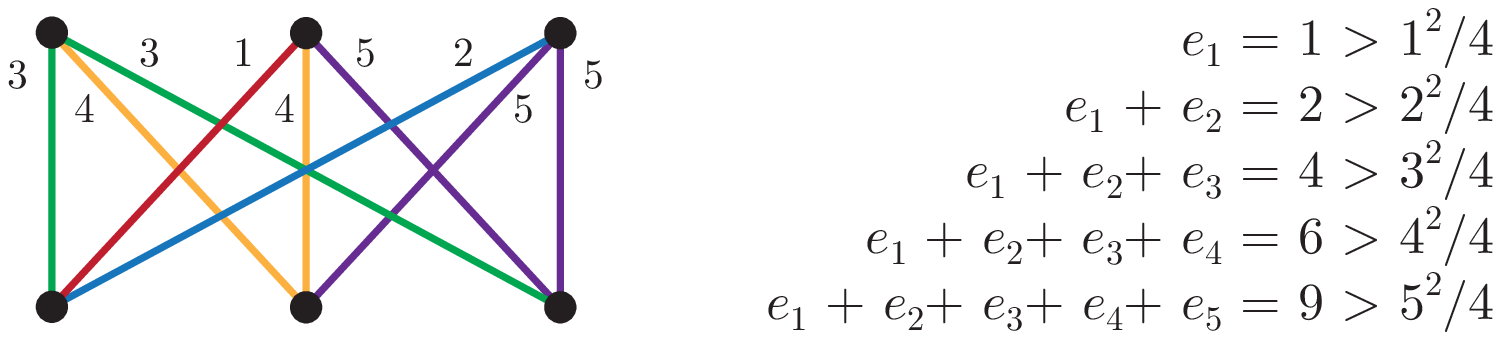}
{An example of Theorem \ref{thm-BrHo01}.}
{fig-110406-1}

Fig.\ref{fig-110406-1} shows an example of Theorem \ref{thm-BrHo01}.
The sum of numbers of edges with $1, 2, \ldots, r$ is more than $r^2/4$
for any color $r$, thus, this graph has a heterochromatic spanning tree.

In the next sections,
we show a generalization of this theorem
and prove it by applying Theorem \ref{thm-Su10-2.3}.

\section{A generalization of Brualdi-Hollingsworth Theorem}

Under the conditions of Theorem \ref{thm-BrHo01},
for any non-empty subset $R \subseteq \C$,
$|E_R(G)| \ge \sum_{i=1}^{|R|} e_i > |R|^2/4$.
On the other hand,
If $|E_R(G)| > |R|^2/4$ for any non-empty subset $R \subseteq \C$,
then
\begin{equation*}
\sum_{i=1}^{r} e_i = \sum_{i=1}^{r} |E_i(G)| = |E_Q(G)| > |Q|^2/4 = r^2/4,
\end{equation*}
for any color $r$ and color subset $Q=\{1,2,3, \ldots , r \} \subseteq \C$.
Thus,
$\sum_{i=1}^{r} e_i > r^2/4$ for any color $r \in \C$
if and only if
$|E_R(G)| > |R|^2/4$ for any non-empty subset $R \subseteq \C$.
Hence, Theorem \ref{thm-BrHo01} implies as follows.

\begin{theorem}[Brualdi and Hollingsworth, 2001 \cite{BrHo01}]
Let $G$ be an edge-colored balanced complete bipartite graph $K_{n,n}$
with a color set $\C = \{1,2,3,$ $\ldots,$ $2n-1\}$.
If $|E_R(G)| > |R|^2/4$ for any non-empty subset $R \subseteq \C$,
then $G$ has a heterochromatic spanning tree.
\label{thm-BrHo01+}
\end{theorem}

In this paper, we generalize this to the following.

\begin{theorem}
Let $G$ be an edge-colored complete bipartite graph $K_{n,m}$
with a color set $\C$.
Let $w$ be a positive integer with $1 \le w \le n+m$,
and $f$ be a function from $\C$ to the set of non-negative integers
such that $\sum_{c \in \C} f(c) \ge n+m-w$.
If $|E_R(G)| > (n+m-w-\sum_{c \in \C \setminus R} f(c))^2/4$
for any non-empty subset $R \subseteq \C$,
then $G$ has an $f$-chromatic spanning forest with $w$ components.
\label{thm-20110510}
\end{theorem}

Theorem \ref{thm-BrHo01+} is a special case
of Theorem \ref{thm-20110510}
with $m=n$, $w=1$, $f(c)=1$ for any color $c$, and $|\C|=2n-1$.

The number of edges of a spanning forest
with $w$ components of $K_{n,m}$ is $n+m-w$.
Thus, in Theorem \ref{thm-20110510},
the condition $\sum_{c \in \C} f(c) \ge n+m-w$ is necessary
for existence of an $f$-chromatic spanning forest with $w$ components.

In Theorem \ref{thm-20110510},
the lower bound of $|E_R(G)|$ is sharp
as follows:
Let $R \subseteq \C$ be a color subset
and $p = n+m-w-\sum_{c \in \C \setminus R} f(c)$.
Let $H$ be a complete bipartite subgraph $K_{\frac{p}{2},\frac{p}{2}}$ of $G$.
Color the edges in $E(H)$ with colors in $R$,
and the edges in $E(G) \setminus E(H)$ with colors in $\C \setminus R$.
Then, $|E_R(G)| = p^2/4 = (n+m-w-\sum_{c \in \C \setminus R} f(c))^2/4$.
(See Figure \ref{fig-110510-1}.)

\fig{0.75\textwidth}{!}{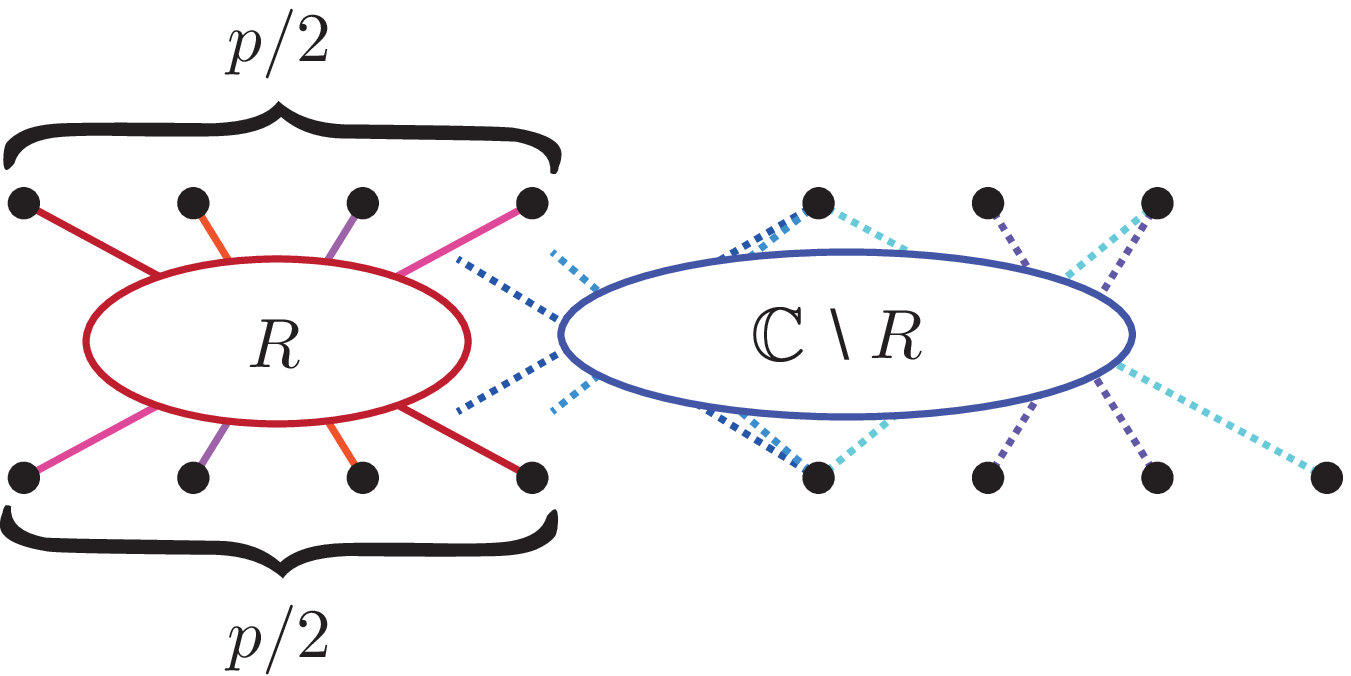}
{A graph $G$ and $R \subseteq \C$
with $|E_R(G)| = (n+m-w-\sum_{c \in \C \setminus R} f(c))^2/4$.}
{fig-110510-1}

Recall that $n+m-w$ is the number of edges of a spanning forest
with $w$ components of $G$,
and $\sum_{c \in \C \setminus R} f(c)$
is the maximum number of edges with colors in $\C \setminus R$ of a desired forest.
Thus, $p$ is the number of edges with colors in $R$ needed in a desired forest.
However, any $p$ edges of $H$ include a cycle because $|V(H)| = p$.
Hence, $G$ has no $f$-chromatic spanning forests with $w$ components,
which implies the lower bound of $|E_R(G)|$ is sharp.

In the next section,
we prove Theorem \ref{thm-20110510}
by applying Theorem \ref{thm-Su10-2.3}.
In order to prove it,
we need the following lemma.

\begin{lemma}
Let $G$ be a bipartite graph of order $N$
that consists of $s$ components.
Then $|E(G)| \le (N-(s-1))^2/4$.
\label{lem-110519-1}
\end{lemma}

\begin{proof}
Take a bipartite graph $G^*$ of order $N$
that consists of $s$ components
so that
\begin{enumerate}
\renewcommand{\labelenumi}{(\arabic{enumi})}
\item $|E(G^*)|$ is maximum, and
\item subject to (1), for the maximum component $D_s$ of $G^*$,
$|V(D_s)|$ is maximum.
\end{enumerate}

By the maximality (1) of $G^*$,
each component of $G^*$ is a complete bipartite graph.
Let $A_s$ and $B_s$ be the partite sets of $D_s$.
We assume $|A_s| \le |B_s|$.

Suppose that
some component $D$ except $D_s$ has at least two vertices.
Let $A$ and $B$ be the partite sets of $D$.
We assume $|A| \le |B|$.
If $|A| > |B_s|$
then $|A_s| \le |B_s| < |A| \le |B|$,
which contradicts that $D_s$ is a maximum component of $G^*$.
Thus, we have $|A| \le |B_s|$.

Let $x$ be a vertex of $B$, where $\deg_G(x) = |A|$.
Let $D' = D-\{ x \}$, $A' = A$, $B' = B-x$,
$A'_s = A_s \cup \{ x \}$, $B'_s = B_s$,
and
$D'_s = (A'_s \cup B'_s, E(D_s) \cup \{xz ~|~ z \in B'_s \})$.
Let $G'^*$ be the resulted graph.
Then, we have
\begin{eqnarray*}
|E(D')|+|E(D'_s)| &=& |E(D)| - \deg_G(x) + |E(D_s)| + |B'_s|\\
                            &=& |E(D)| + |E(D_s)| + |B_s| - |A|\\
                            & \ge & |E(D)| + |E(D_s)|,
\end{eqnarray*}
which implies $|E(G'^*)|=|E(G^*)|$ by the condition (1).
However, that contradicts the maximality (2)
because $|D'_s| \ge |D_s| + 1$.
Hence, every component except $D_s$ has exactly one vertex,
which implies that
$|V(D_s)|=N-(s-1)$.

Suppose that $|B_s|-|A_s| \ge 2$.
Let $x$ be a vertex of $B_s$, where $\deg_G(x) = |A_s|$.
Let $D'_s = (V(D_s), E(D_s - x) \cup \{xz ~|~ z \in B_s -x \})$.
Then, $D'_s$ is a complete bipartite graph, and we have
\begin{eqnarray*}
|E(D'_s)| &=& |E(D_s)| - \deg_G(x) + |B_s - x|\\
               &=& |E(D_s)| - |A_s| + |B_s| - 1\\
               & \ge & |E(D_s)| + 1,
\end{eqnarray*}
which contradicts the maximality (1).
Hence, $|B_s|-|A_s| \le 1$.

Therefore,
\begin{eqnarray*}
|E(G)| & \le & |E(G^*)| = |E(D_s)| = |A_s||B_s| \\
           & = & \lfloor (N-(s-1))/2 \rfloor \lceil (N-(s-1))/2 \rceil \\
           & \le & (N-(s-1))^2/4.
\end{eqnarray*}
\end{proof}

\section{Proof of Theorem \ref{thm-20110510}}

Suppose that
$G$ has no $f$-chromatic spanning forests with $w$ components.
By Theorem \ref{thm-Su10-2.3},
there exists a color set $R \subseteq \C$
such that
\begin{equation}
\omega(G-E_R(G)) > w+\sum_{c \in R}f(c).
\label{eq-110519-1}
\end{equation}

Let $s = \omega(G-E_R(G))$.
Let $D_1, D_2, \ldots, D_s$ be the components of $G-E_R(G)$,
and $q$ be the number of edges of $G$ between these distinct components.
Note that,
the colors of these $q$ edges are only in $R$.

If $R=\C$ then
\begin{equation*}
s = \omega(G-E_{\C}(G)) > w+\sum_{c \in \C}f(c)
\ge w+n+m-w=n+m=|V(G)|.
\end{equation*}
by the assumption of Theorem \ref{thm-20110510}.
This contradicts that $s \le |V(G)|$.
Thus,
we can assume $R \ne \C$, namely, $\C \setminus R \ne \emptyset$.
Hence,
by the assumption of Theorem \ref{thm-20110510},
\begin{equation*}
|E_{\C \setminus R}(G)|
> (n+m-w-\sum_{c \in \C \setminus (\C \setminus R)} f(c))^2/4
= (n+m-w-\sum_{c \in R} f(c))^2/4.
\end{equation*}

Therefore, we have
\begin{equation}
q \le |E_R(G)| = |E(G)| - |E_{\C \setminus R}(G)|
<|E(G)| - (n+m-w-\sum_{c \in R} f(c))^2/4.
\label{eq-110519-2}
\end{equation}

On the other hand,
\begin{eqnarray*}
q &=& |E(G)| - |E(D_1) \cup E(D_2) \cup \cdots \cup E(D_s)|.
\end{eqnarray*}
By Lemma \ref{lem-110519-1},
$|E(D_1) \cup E(D_2) \cup \cdots \cup E(D_s)| \le (n+m-(s-1))^2/4$.
Thus,
since $s = \omega(G-E_R(G)) \ge w+1+\sum_{c \in R}f(c)$
by (\ref{eq-110519-1}),
we have
\begin{eqnarray*}
q &=& |E(G)| - |E(D_1) \cup E(D_2) \cup \cdots \cup E(D_s)|\\
  &\ge & |E(G)| - (n+m-(s-1))^2/4\\
  &\ge & |E(G)| - (n+m-(w+1+\sum_{c \in R}f(c)-1))^2/4\\
  &=& |E(G)| - (n+m-w-\sum_{c \in R} f(c))^2/4,
\end{eqnarray*}
which contradicts (\ref{eq-110519-2}).
Consequently,
the graph $G$ has an $f$-chromatic spanning forest
with $w$ components.



\end{document}